\documentclass{myclass}

\title{A rigidity result for axisymmetric toric Ricci solitons}
\author[S. Zhang]{Shiqiao Zhang}
\dateposted{October 2024}

\begin{document}

\begin{abstract}
    We examine a non-axisymmetric perturbation
    of a family of axisymmetric toric Einstein manifolds and Ricci solitons
    studied in Firester--Tsiamis (2024).
    We establish a rigidity result
    stating that these axisymmetric Ricci solitons
    do not admit constant-angle non-axisymmetric perturbations
    except for conformally flat cases.
    For these new cases,
    our result leads to an explicit description of the Einstein metrics
    and a classification of the Ricci solitons
    under a volume-collapsing ansatz.
\end{abstract}

\maketitle

\vspace*{-0.3in}

\section{Introduction}

Ricci solitons are the self-similar solutions to the Ricci flow \(\partial_t g = -2\Ric(g)\).
Characterized by the quasi-Einstein metric equation
\(\Ric(g) + \frac{1}{2} \Lie_V g = \lambda g\),
Ricci solitons model the formation of singularities in the Ricci flow;
see \cite{Cao09} for a survey of the role of Ricci solitons
in the singularity study of the Ricci flow.
In \cite{FT24}, Firester and Tsiamis
investigated a family of axisymmetric toric Ricci solitons in dimension four.
They classified them based on geometric properties
motivated by known families of Ricci solitons
and studied the geometric behavior of new solutions.
We show that these axisymmetric Ricci solitons are rigid
with respect to constant-angle non-axisymmetric perturbations;
specifically, no such perturbations exist
unless the Ricci soliton is conformally flat.
This rigidity comes from an off-diagonal obstruction in the Ricci soliton equation
that disappears when the metric is axisymmetric.
Based on this rigidity result,
we transfer a classification by Firester and Tsiamis
to constant-angle non-axisymmetric toric Ricci solitons.

Understanding the behavior of Ricci solitons in dimensions four and above
is essential for the singularity study
of the Ricci flow in higher dimensions,
as a foundational work of Bamler \cite{Bam18}
showed that the singular sets of Ricci flows with bounded scalar curvature
have codimension at least four.
There is no known classification of four-dimensional Ricci solitons
except in the Kähler setting,
where Bamler, Cifarelli, Conlon, and Deruelle \cite{BCCD24}
completed the classification of shrinking
Kähler--Ricci solitons in complex dimension two.
In \cite{App23}, Appleton showed that four-dimensional Ricci solitons
can exhibit multiple blow-up behaviors in the same soliton,
including orbifold singularities,
calling for the research of Ricci solitons with singularities or incompleteness.
For example, Bamler, Chow, Deng, Ma, and Zhang \cite{BCD+22}
classified the tangent flows at infinity
of steady soliton singularity models in dimension four;
Alexakis, Chen, and Fournodavlos \cite{ACF15} constructed a new family
of spherically symmetric singular Ricci solitons in all dimensions
and studied their stability under spherically symmetric perturbations;
and Hui \cite{Hui24} showed the existence of
singular rotationally symmetric steady and expanding gradient Ricci solitons
of the form
\(g = {da^2}/{h(a^2)} + a^2 g_{\Sphere_n}\)
in dimension \(n + 2 \ge 4\).

Many known examples of four-dimensional Ricci solitons
are cohomogeneity two metrics,
which are metrics that admit an isometric torus action.
In \cite{FT24}, Firester and Tsiamis
studied a family of axisymmetric toric Ricci solitons given locally by
\begin{equation} \label{eq:g.f.t}
    g = \frac{1}{q(x, y)^2} \biggl(
        \frac{dx^2}{A(x)} + \frac{dy^2}{B(y)}
        + A(x) \, ds^2 + B(y) \, dt^2
    \biggr). \tag{\(\star\)}
\end{equation}
The axisymmetric property means that the torus fibers with coordinates \((s, t)\)
are orthogonal, so the metric is diagonal.
Firester and Tsiamis classified Ricci solitons of the form \eqref{eq:g.f.t}
whose base metric
\begin{equation} \label{eq:g.base}
    \frac{1}{q(x, y)^2} \biggl(
        \frac{dx^2}{A(x)} + \frac{dy^2}{B(y)}
    \biggr)
\end{equation}
is conformally cylindrical,
which means that the function \(q\) is homogeneous of degree \(1\).
This volume-collapsing ansatz
ensures that the metric \eqref{eq:g.f.t} is asymptotically cylindrical
and allows reducing the Ricci soliton equation to an ordinary differential equation.
See \cref{sec:f.t} for a summary
of the results and methods of Firester and Tsiamis
that we make use of.

We consider a non-axisymmetric perturbation
\begin{equation} \label{eq:g}
    g = \frac{1}{q(x, y)^2} \biggl(
        \frac{dx^2}{A(x)} + \frac{dy^2}{B(y)}
        + A(x) \, ds^2 + B(y) \, dt^2 + 2 \sqrt{A(x) B(y)} \cos \theta \, ds \, dt
    \biggr) \tag{\(\star\star\)}
\end{equation}
of the metric \eqref{eq:g.f.t},
in which the torus fibers
are placed at a constant angle \(0 < \theta < \pi/2\);
it would be interesting to understand Ricci solitons
where \(\theta\) varies in the base \((x, y)\).
The metric \eqref{eq:g}
recovers the axisymmetric metric \eqref{eq:g.f.t} when \(\theta\) is identically \(\pi/2\).
In \cref{thm:main},
we show a rigidity result
for non-axisymmetric Ricci solitons of the form \eqref{eq:g},
which requires that \(A\) and \(B\) be constants
in addition to the same system of differential equations
\labelcref{eq:diff.eq.f.t.1,eq:diff.eq.f.t.2,eq:diff.eq.f.t.3}
arising from the metric \eqref{eq:g.f.t}.
Here, the subscripts on \(q\) denote partial derivatives,
and function arguments are omitted for brevity.

\begin{theorem} \label{thm:main}
    A non-axisymmetric metric \eqref{eq:g}
    is a Ricci soliton if and only if \(A\) and \(B\) are constants
    and there are functions
    \(S^x(x, y)\) and \(S^y(x, y)\) satisfying the system of equations
    \begin{gather}
        \partial_y S^x + \partial_x S^y = 4 q_x q_y,
        \qquad
        \partial_x S^x = 2 q_x^2,
        \qquad
        \partial_y S^y = 2 q_y^2, \label{eq:diff.eq.S} \\
        \lambda = q_{xx} A + q_{yy} B
            - \frac{q_x^2 A + q_y^2 B}{q}
            - \frac{S^x q_x A + S^y q_y B}{q^2}.
            \label{eq:diff.eq.lambda}
    \end{gather}
\end{theorem}
In particular, any such metric is conformally flat.
It is only in the axisymmetric case \(\theta = \pi/2\)
that there are many Ricci solitons,
which are therefore rigid with respect to constant-angle non-axisymmetric perturbations.
\Cref{thm:main} leads to an explicit description of Einstein metrics of the form \eqref{eq:g},
which are characterized by the Einstein metric equation \(\Ric(g) = \lambda g\).

\begin{corollary} \label{cor:einstein}
    A non-axisymmetric metric \eqref{eq:g}
    is an Einstein metric if and only if \(A\) and \(B\) are constants
    and \(q(x, y) = ax + by + c\) for some constants \(a\), \(b\), and \(c\),
    with \(\lambda = -3 (a^2 A + b^2 B)\)
    in the Einstein metric equation.
\end{corollary}

Under the conformally cylindrical ansatz,
the additional constraint that \(A\) and \(B\) are constants
narrows down the classification of Ricci solitons of the form \eqref{eq:g}.

\begin{corollary} \label{cor:cc}
    A non-axisymmetric metric \eqref{eq:g}
    whose base metric \eqref{eq:g.base} is conformally cylindrical
    is one of
    \begin{enumerate}
    \item an Einstein metric as described in \cref{cor:einstein};
    \item a product of two Ricci soliton surfaces
        with the same constant \(\lambda\)
        defined in the Ricci soliton equation; or
    \item a product of a flat surface with an Einstein surface, given by
        \[
            g = \frac{1}{f(x)^2} \biggl(
                \frac{dx^2}{A} + \frac{dy^2}{B} + A \, ds^2 + B \, dt^2
            \biggr)
        \]
        up to translating, rescaling, and swapping \(x\) and \(y\),
        where \(f \in \{\cosh, \sinh, \sin\}\),
        and \(A\) and \(B\) are constants.
    \end{enumerate}
\end{corollary}

\section{The axisymmetric case} \label{sec:f.t}

In Theorem~1.1 in \cite{FT24},
Firester and Tsiamis classified Ricci solitons of the form \eqref{eq:g.f.t}
whose base metric \eqref{eq:g.base} is conformally cylindrical.
We summarize this classification in \cref{thm:f.t.1.1.1}.
Recall that the base metric \eqref{eq:g.base} being conformally cylindrical
means that the function \(q\) is homogeneous of degree \(1\).

\begin{theorem}[Firester--Tsiamis] \label{thm:f.t.1.1.1}
    When the base metric \eqref{eq:g.base} is conformally cylindrical,
    the metric \eqref{eq:g.f.t} is one of
    \begin{enumerate}
    \item an Einstein metric, specifically the Riemannian Schwarzschild metric
        \begin{equation} \label{eq:g.schwarzschild}
            g = \frac{1}{x^2} \biggl(
                \frac{dx^2}{A(x)} + \frac{dy^2}{B(y)}
                + A(x) \, ds^2 + B(y) \, dt^2
            \biggr)
        \end{equation}
        for some cubic \(A\) and quadratic \(B\),
        possibly degenerate,
        or the Plebański--Demiański metric
        \begin{equation} \label{eq:g.p.d}
            g = \frac{1}{(x \pm y)^2} \biggl(
                \frac{dx^2}{a_0 - P(x)} + \frac{dy^2}{b_0 + P(y)}
                + (a_0 - P(x)) \, ds^2 + (b_0 + P(y)) \, dt^2
            \biggr)
        \end{equation}
        for some constants \(a_0\) and \(b_0\)
        and possibly degenerate cubic \(P\),
        up to translating, rescaling, and swapping \(x\) and \(y\);
    \item a product of two Ricci soliton surfaces
        with the same constant \(\lambda\)
        defined in the Ricci soliton equation;
    \item a product of a flat surface with an Einstein surface;
    \item a metric locally conformal to \(\Sphere^2 \times \HH^2\) or \(\RR^4\),
        given explicitly by
        \[
            g = \frac{1}{y^2 f(x/y)^2} \biggl(
                \frac{dx^2}{a_0 - x^2} + \frac{dy^2}{b_0 + y^2}
                + (a_0 - x^2) \, ds^2 + (b_0 + y^2) \, dt^2
            \biggr)
        \]
        for some constants \(a_0\) and \(b_0\)
        and solution \(f(z)\) to the ordinary differential equation
        \begin{align*}
            0 = {}
            & {-3} (a_0 (f')^2 + b_0 (f - z f')^2)^2 
                - \lambda (a_0 (f')^2 + b_0 (f - z f')^2) \nonumber \\
            & + (4 a_0 - b_0 z^2) b_0 f^3 f''
                + (a_0 + b_0 z^2)^2 f f''((f')^2 - f f'')
                + \lambda (a_0 + b_0 z^2) f f'' \nonumber \\
            & - (a_0 + b_0 z^2) (b_0 z f - (a_0 + b_0 z^2) f') f^2 f^{(3)};
        \end{align*}
        or
    \item a steady singular Ricci soliton given by
        \[
            g = \frac{1}{x^{2 \alpha} y^{2(1 - \alpha)}} \biggl(
                  \frac{dx^2}{F_{\alpha, c, k_1}(x)}
                + \frac{dy^2}{F_{1 - \alpha, -c, k_2}(y)}
                + F_{\alpha, c, k_1}(x) \, ds^2
                + F_{1 - \alpha, -c, k_2}(y) \, dt^2
            \biggr)
        \]
        where
        \[
            F_{\alpha, c, k}(t) \coloneqq c t^2 + k t^{\frac{2 \alpha^2}{2 \alpha - 1} + 1}
        \]
        for some constants \(c\), \(k_1\), \(k_2\), and \(\alpha \neq 1/2\).
    \end{enumerate}
\end{theorem}

To establish the classification in \cref{thm:f.t.1.1.1},
Firester and Tsiamis introduced auxiliary functions \eqref{eq:s.f.t}
and reduced the Ricci soliton equation
into the system of differential equations
\labelcref{eq:diff.eq.f.t.1,eq:diff.eq.f.t.2,eq:diff.eq.f.t.3}.
This result is stated in Corollary~{2.4} of \cite{FT24}.
Based on this result, they specialized to the case
where the base metric \eqref{eq:g.base} is conformally cylindrical
and further reduced the Ricci soliton equation
to an ordinary differential equation.

\begin{theorem}[Firester--Tsiamis] \label{lem:f.t.2.4}
    When the base metric \eqref{eq:g.base} is conformally cylindrical,
    the metric \eqref{eq:g.f.t} is a Ricci soliton if and only if
    there are functions
    \(S^x(x, y)\) and \(S^y(x, y)\) satisfying the system of differential equations
    \begin{gather}
        \partial_y S^x + \partial_x S^y = 4 q_x q_y,
        \qquad
        \partial_x S^x = 2 q_x^2,
        \qquad
        \partial_y S^y = 2 q_y^2,
            \label{eq:diff.eq.f.t.1} \\
        w \coloneqq A'' - \frac{S^x}{q^2} A' = B'' - \frac{S^y}{q^2} B',
            \label{eq:diff.eq.f.t.2} \\
        - \frac{\lambda}{q}
            = \frac{1}{2} w
            - \frac{q_x A' + q_{xx} A + q_y B' + q_{yy} B}{q}
            + \frac{q_x^2 A + q_y^2 B}{q^2}
            + \frac{S^x q_x A + S^y q_y B}{q^3}
            \label{eq:diff.eq.f.t.3}.
    \end{gather}
\end{theorem}

Firester and Tsiamis's proof of \cref{lem:f.t.2.4}
uses the following result,
stated and proved as Claim~2.3 in \cite{FT24}.

\begin{lemma}[Firester--Tsiamis] \label{lem:f.t.2.3}
    If the metric \eqref{eq:g.f.t} satisfies
    the Ricci soliton equation \(\Ric(g) + \frac{1}{2} \Lie_V g = \lambda g\),
    then there is some smooth vector field of the form
    \(\tilde{V} = \tilde{V}^x(x, y) \, \partial_x + \tilde{V}^y(x, y) \, \partial_y\)
    such that \(\Ric(g) + \frac{1}{2} \Lie_{\tilde{V}} g = \lambda g\).
\end{lemma}

We now review Firester and Tsiamis's proof of \cref{lem:f.t.2.4},
which forms the basis of the proof of our rigidity result in \cref{sec:proof.main}.

\begin{proof}[Proof of \cref{lem:f.t.2.4}]
    The Ricci soliton equation
    can be rewritten as
    \begin{equation} \label{eq:ricci.soliton.f.t}
        \Lambda \coloneqq g^{-1} (\Ric(g) + \tfrac{1}{2} \Lie_V g) = \lambda I,
    \end{equation}
    where \(I\) denotes the identity tensor.
    The Ricci tensor of the metric \eqref{eq:g.f.t} is given by
    \begin{align*}
        \Ric_{xx} &= - \frac{1}{A} \biggl(
            \frac{A''}{2} - \frac{2 q_x A' + 3 q_{xx} A + q_y B' + q_{yy} B}{q}
            + \frac{3 (q_x^2 A + q_y^2 B)}{q^2}
        \biggr), \\
        \Ric_{xy} &= \frac{2 q_{xy}}{q}, \\
        \Ric_{ss} &= - A \biggl(
            \frac{A''}{2} - \frac{2 q_x A' + 3 q_{xx} A + q_y B' + q_{yy} B}{q}
            + \frac{3 (q_x^2 A + q_y^2 B)}{q^2}
        \biggr),
    \end{align*}
    and the symmetry \((x, A) \leftrightarrow (y, B)\),
    where all off-diagonal entries of \(\Ric\) are zero
    except \(\Ric_{xy}\) and \(\Ric_{yx}\).
    For the metric \eqref{eq:g.f.t},
    \cref{lem:f.t.2.3} allows us to assume without loss of generality that
    \(V = V^x(x, y) \, \partial_x + V^y(x, y) \, \partial_y\)
    for some functions \(V^x\) and \(V^y\).
    Then, the Lie derivative \(\Lie_V g\) is given by
    \begin{align*}
        (\Lie_V g)_{xx}
            &= \frac{2 (\partial_x V^x) A - V^x A'}{q^2 A^2}
            - \frac{2 (V^x q_x + V^y q_y)}{q^3 A}, \\
        (\Lie_V g)_{xy}
            &= \frac{(\partial_x V^y) A + (\partial_y V^x) B}{q^2 AB}, \\
        (\Lie_V g)_{ss}
            &= \frac{V^x A'}{q^2} - \frac{2 A (V^x q_x + V^y q_y)}{q^3},
    \end{align*}
    and the same symmetry \((x, A) \leftrightarrow (y, B)\),
    where all off-diagonal entries of \(\Lie_V g\) are zero
    except \((\Lie_V g)_{xy}\) and \((\Lie_V g)_{yx}\).
    The Ricci soliton equation \eqref{eq:ricci.soliton.f.t}
    is equivalent to the system of equations
    \({\Lambda^x}_y = {\Lambda^y}_x = 0\)
    and
    \({\Lambda^x}_x = {\Lambda^y}_y = {\Lambda^s}_s = {\Lambda^t}_t = \lambda\).
    To reduce these equations,
    we introduce the auxiliary functions
    \begin{equation} \label{eq:s.f.t}
        S^x \coloneqq 2 q q_x + \frac{V^x}{A}
        \qquad \text{and} \qquad
        S^y \coloneqq 2 q q_y + \frac{V^y}{B}.
    \end{equation}
    Calculation shows that
    all off-diagonal entries of \(\Lambda\) are identically zero
    except \({\Lambda^x}_y\) and \({\Lambda^y}_x\).
    Both \({\Lambda^x}_y = 0\) and \({\Lambda^y}_x = 0\) reduce to
    \(\partial_y S^x + \partial_x S^y = 4 q_x q_y\).
    Equating pairs of diagonal entries of \(\Lambda\),
    the equations \({\Lambda^x}_x = {\Lambda^s}_s\)
    and \({\Lambda^y}_y = {\Lambda^t}_t\)
    simplify to
    \(\partial_x S^x = 2 q_x^2\) and \(\partial_y S^y = 2 q_y^2\)
    respectively,
    while the equation \({\Lambda^x}_x = {\Lambda^y}_y\) reduces to
    \[
        w(x, y) \coloneqq A'' - \frac{S^x}{q^2} A' = B'' - \frac{S^y}{q^2} B'.
    \]
    Finally, the equation \({\Lambda^x}_x = \lambda\) simplifies to
    \[
        - \frac{\lambda}{q}
            = \frac{1}{2} w
            - \frac{q_x A' + q_{xx} A + q_y B' + q_{yy} B}{q}
            + \frac{q_x^2 A + q_y^2 B}{q^2}
            + \frac{S^x q_x A + S^y q_y B}{q^3},
    \]
    thereby reducing the Ricci soliton equation \eqref{eq:ricci.soliton.f.t}
    to the system of differential equations
    \labelcref{eq:diff.eq.f.t.1,eq:diff.eq.f.t.2,eq:diff.eq.f.t.3}.
\end{proof}

\section{The non-axisymmetric case} \label{sec:proof}

\subsection{Rigidity result} \label{sec:proof.main}
\cref{thm:main} shows that
Ricci solitons of the form \eqref{eq:g}
are constrained by an obstacle
arising from the \({\Lambda^s}_t = 0\) component
of the Ricci soliton equation \eqref{eq:ricci.soliton}.
This obstacle disappears when \(\theta = \pi/2\),
allowing a much larger family of Ricci solitons.
To prove \cref{thm:main},
we require an analogue of \cref{lem:f.t.2.3} for the metric \eqref{eq:g}.
The same proof that Firester and Tsiamis provided
for Claim~2.3 in \cite{FT24} also applies to this analogue; we provide a different constructive proof.

\begin{lemma} \label{lem:v}
    If the metric \eqref{eq:g} satisfies
    the Ricci soliton equation \(\Ric(g) + \frac{1}{2} \Lie_V g = \lambda g\),
    then there is some smooth vector field of the form
    \(\tilde{V} = \tilde{V}^x(x, y) \, \partial_x + \tilde{V}^y(x, y) \, \partial_y\)
    such that \(\Ric(g) + \frac{1}{2} \Lie_{\tilde{V}} g = \lambda g\).
\end{lemma}

\begin{proof}
    Suppose that the metric \(g\) of the form \eqref{eq:g}
    satisfies the Ricci soliton equation
    with the vector field \(V\).
    For the Lie derivative \(\Lie_V g\),
    we may assume \(V^s = V^t = 0\)
    since the coefficients of \(g\) depend only on \(x\) and \(y\).
    Calculation shows that the Ricci tensor \(\Ric\) of the toric metric \(g\) is also toric,
    namely \(\Ric = \Ric_{\mathrm{base}} \oplus \Ric_{\mathrm{torus}}\)
    where the metrics \(\Ric_{\mathrm{base}}\) over \((x, y)\)
    and \(\Ric_{\mathrm{torus}}\) over \((s, t)\)
    both have coefficients depending only on \(x\) and \(y\).
    Hence, \(\Lie_V g = 2(\lambda g - \Ric)\) is also toric,
    meaning that \(\Lie_V g = \Lie_{\tilde{V}} g\) for the vector field
    \[
        \tilde{V}(x, y) = \frac{1}{\abs{\Omega}} \int_{\Omega} V(x, y, s, t) \, ds \wedge dt
    \]
    obtained by averaging \(V\) over a region \(\Omega\) in the coordinates \((s, t)\)
    constituting one period of the toric fibers.
\end{proof}

The conditions on \(q\), \(A\), \(B\), \(S^x\), and \(S^y\) in \cref{thm:main}
are the same as those in \cref{lem:f.t.2.4}
with the additional constraint
that \(A\) and \(B\) are constants,
so \(A' = B' = 0\) and \(w(x, y) = 0\).
The value of the constant \(0 < \theta < \pi/2\)
cancels out in the calculations
and does not appear in the resulting conditions.
When \(\theta = \pi/2\) as in \cref{lem:f.t.2.4},
however, the equations \({\Lambda^s}_t = 0\) and \({\Lambda^t}_s = 0\)
in the proof below hold identically instead,
so the constraint that \(A\) and \(B\) are constants is lifted,
resulting in a much wider range of Ricci solitons.

\begin{proof}[Proof of \cref{thm:main}]
    Recall that the Ricci soliton equation
    can be rewritten as
    \begin{equation} \label{eq:ricci.soliton}
        \Lambda \coloneqq g^{-1} (\Ric(g) + \tfrac{1}{2} \Lie_V g) = \lambda I
    \end{equation}
    where \(I\) denotes the identity tensor.
    The Ricci tensor of the metric \eqref{eq:g} is given by
    \begin{align*}
        \Ric_{xx} &=  - \frac{1}{A} \biggl(
                \frac{A''}{2} - \frac{2 q_x A' + 3 q_{xx} A + q_y B' + q_{yy} B}{q}
                + \frac{3 (q_x^2 A + q_y^2 B)}{q^2} + \frac{(A')^2 \cot^2 \theta}{8A}
            \biggr), \\
        \Ric_{xy} &= \frac{2 q_{xy}}{q} + \frac{A' B' \cot \theta}{8 A B}, \\
        \Ric_{ss} &= - A \biggl(
                \frac{A''}{2} - \frac{2 q_x A' + 3 q_{xx} A + q_y B' + q_{yy} B}{q} \\
                &\qquad + \frac{3 (q_x^2 A + q_y^2 B)}{q^2} - \frac{(B (A')^2 + A (B')^2) \cot^2 \theta}{8AB}
            \biggr), \\
        \Ric_{st} &= - \sqrt{AB} \cos \theta \biggl(
                \frac{A'' + B''}{4}
                - \frac{3 (q_x A' + q_y B') + 2 (q_{xx} A + q_{yy} B)}{2 q}
                + \frac{3 (q_x^2 A + q_y^2 B)}{q^2}
            \biggr) \\
            &\qquad + \frac{(B (A')^2 + A (B')^2) \cos \theta}{8 \sqrt{AB} \sin^2 \theta},
    \end{align*}
    and the symmetry \((x, A) \leftrightarrow (y, B)\),
    where all off-diagonal entries of \(\Ric\) are zero
    except \(\Ric_{xy}\), \(\Ric_{yx}\), \(\Ric_{st}\), and \(\Ric_{ts}\).
    For the metric \eqref{eq:g},
    \cref{lem:v} allows us to assume without loss of generality that
    \(V = V^x(x, y) \, \partial_x + V^y(x, y) \, \partial_y\)
    for some functions \(V^x\) and \(V^y\).
    Then, the Lie derivative \(\Lie_V g\) is given by
    \begin{align*}
        (\Lie_V g)_{xx}
            &= \frac{2 (\partial_x V^x) A - V^x A'}{q^2 A^2}
            - \frac{2 (V^x q_x + V^y q_y)}{q^3 A}, \\
        (\Lie_V g)_{xy}
            &= \frac{(\partial_x V^y) A + (\partial_y V^x) B}{q^2 AB}, \\
        (\Lie_V g)_{ss}
            &= \frac{V^x A'}{q^2} - \frac{2 A (V^x q_x + V^y q_y)}{q^3}, \\
        (\Lie_V g)_{st}
            &= \frac{\cos \theta}{\sqrt{AB}} \biggl(
                \frac{V^x A' B + V^y B' A}{2 q^2}
                - \frac{2 A B (V^x q_x + V^y q_y)}{q^3}
            \biggr),
    \end{align*}
    and the same symmetry \((x, A) \leftrightarrow (y, B)\),
    where all off-diagonal entries of \(\Lie_V g\) are zero
    except \((\Lie_V g)_{xy}\), \((\Lie_V g)_{yx}\),
    \((\Lie_V g)_{st}\), and \((\Lie_V g)_{ts}\).
    The Ricci soliton equation \eqref{eq:ricci.soliton}
    is equivalent to the system of equations
    \[
        {\Lambda^x}_y = {\Lambda^y}_x = {\Lambda^s}_t = {\Lambda^t}_s = 0
        \qquad \text{and} \qquad
        {\Lambda^x}_x = {\Lambda^y}_y = {\Lambda^s}_s = {\Lambda^t}_t = \lambda.
    \]
    Given \(0 < \theta < \pi/2\),
    the equations \({\Lambda^s}_t = 0\) and \({\Lambda^t}_s = 0\)
    simplify to
    \begin{align*}
        2 (A' B V^x - A B' V^y)
        + 4 q A B (A' q_x - B' q_y)
        + q^2 (A'^2 B + A B'^2 - 2 A B (A'' - B'')) &= 0 \\
        \intertext{and}
        2 (A' B V^x - A B' V^y)
        + 4 q A B (A' q_x - B' q_y)
        - q^2 (A'^2 B + A B'^2 + 2 A B (A'' - B'')) &= 0
    \end{align*}
    respectively.
    Subtracting these equations and dividing through by \(q^2 A B\) results in
    \[
        \frac{A'^2}{A} + \frac{B'^2}{B} = 0.
    \]
    This forces \(A' = B' = 0\) since \(A > 0\) and \(B > 0\),
    so \(A\) and \(B\) are constants.
    Under this condition, both \({\Lambda^x}_y = 0\) and \({\Lambda^y}_x = 0\)
    simplify to
    \[
        \frac{\partial_y V^x}{A} + \frac{\partial_x V^y}{B} + 4 q q_{xy} = 0,
    \]
    which is equivalent to
    \(\partial_y S^x + \partial_x S^y = 4 q_x q_y\)
    under the introduction of the same auxiliary functions
    \begin{equation} \label{eq:s}
        S^x \coloneqq 2 q q_x + \frac{V^x}{A}
        \qquad \text{and} \qquad
        S^y \coloneqq 2 q q_y + \frac{V^y}{B}
    \end{equation}
    from \eqref{eq:s.f.t}.
    Further calculation shows that
    the equations \({\Lambda^x}_x = {\Lambda^s}_s\) and \({\Lambda^y}_y = {\Lambda^t}_t\)
    simplify to
    \(\partial_x S^x = 2 q_x^2\) and \(\partial_y S^y = 2 q_y^2\)
    respectively,
    whereas \({\Lambda^s}_s = {\Lambda^t}_t\) holds identically.
    Finally, the equation \({\Lambda^x}_x = \lambda\) simplifies to
    \[
        \lambda = q_{xx} A + q_{yy} B
            - \frac{q_x^2 A + q_y^2 B}{q}
            - \frac{S^x q_x A + S^y q_y B}{q^2}.
    \]
    Hence, we have reduced the Ricci soliton equation \eqref{eq:ricci.soliton}
    to the system of equations \labelcref{eq:diff.eq.S,eq:diff.eq.lambda}.
\end{proof}

\subsection{Corollaries}

The description of Einstein metrics of the form \eqref{eq:g}
follows immediately from \cref{thm:main}.

\begin{proof}[Proof of \cref{cor:einstein}]
    If the metric \eqref{eq:g} is an Einstein metric,
    then we have \(V = 0\) in \cref{thm:main},
    so the auxiliary functions \eqref{eq:s} are
    \(S^x = 2 q q_x\) and \(S^y = 2 q q_y\)
    respectively.
    Hence, the system of equations \eqref{eq:diff.eq.S}
    reduces to
    \(q_{xx} = q_{xy} = q_{yy} = 0\),
    so \(q = ax + by + c\) for some constants \(a\), \(b\), and \(c\).
    Finally, \eqref{eq:diff.eq.lambda} becomes
    \(\lambda = -3(a^2 A + b^2 B)\).
\end{proof}

Finally, we classify Ricci solitons of the form \eqref{eq:g}
whose base metric \eqref{eq:g.base}
is conformally cylindrical.

\begin{proof}[Proof of \cref{cor:cc}]
    For a Ricci soliton \(g\) of the form \eqref{eq:g},
    \cref{thm:main} implies that
    the corresponding diagonal metric \eqref{eq:g.f.t}
    is a Ricci soliton classified in \cref{thm:f.t.1.1.1}.
    There are only three cases
    where \(A\) and \(B\) can be constants, namely
    \begin{enumerate}
    \item when \(g\) is an Einstein metric
        as described in \cref{cor:einstein};
    \item when \(g\) is a product of two Ricci soliton surfaces
        with the same constant \(\lambda\)
        defined in the Ricci soliton equation;
        and
    \item when \(g\) is a product of a flat surface with an Einstein surface.
    \end{enumerate}
    We compute the form of \(g\) explicitly in the last case.
    Suppose that \(g = g_{\mathrm{Einstein}} \oplus g_{\mathrm{flat}}\)
    where
    \[
        g_{\mathrm{Einstein}} = \frac{1}{q(x, y)^2} \biggl(
            \frac{dx^2}{A} + A \, ds^2
        \biggr)
    \]
    is an Einstein metric and
    \[
        g_{\mathrm{flat}} = \frac{1}{q(x, y)^2} \biggl(
            \frac{dy^2}{B} + B \, dt^2
        \biggr)
    \]
    is a flat metric.
    The Ricci tensor of \(g_{\mathrm{flat}}\) is diagonal with entries
    \[
        \Ric_{\mathrm{flat}, yy} = \frac{q q_{yy} - q_y^2}{q^2}
        \qquad \text{and} \qquad
        \Ric_{\mathrm{flat}, tt} = \frac{B^2 (q q_{yy} - q_y^2)}{q^2},
    \]
    so the flatness condition reduces to
    \(q q_{yy} - q_y^2 = 0\),
    whose solution is
    \(q(x, y) = h(x) e^{y f(x)}\)
    for some functions \(f\) and \(h\).
    Similar calculation shows that the Einstein condition
    on \(g_{\mathrm{Einstein}}\) reduces to
    \(q q_{xx} - q_x^2 = {\lambda}/{A}\),
    which becomes
    \begin{equation} \label{eq:cc.1}
        e^{2y f} (h (x h f'' + h'') - (h')^2) = \frac{\lambda}{A}.
    \end{equation}
    Since the left-hand side is constant for all \(x\) and \(y\),
    its partial derivative with respect to \(y\) must be zero,
    which results in
    \begin{equation} \label{eq:cc.2}
        f'' = - \frac{2 f (h h'' - (h')^2)}{h^2 (2 y f + 1)}.
    \end{equation}
    Substituting \eqref{eq:cc.2} into \eqref{eq:cc.1} results in
    \[
        \frac{e^{2 y f} (h h'' - (h')^2)}{2 y f + 1} = \frac{\lambda}{A}.
    \]
    Setting the partial derivative of the left-hand side
    with respect to \(y\) equal to zero again results in \(y f(x) = 0\),
    which must hold for all \(x\) and \(y\).
    Hence, \(f\) is identically zero,
    and \eqref{eq:cc.1} reduces to
    \(h h'' - (h')^2 = {\lambda}/{A}\),
    whose solutions, up to translating and rescaling \(x\), are
    \(h(x) = \pm \sqrt{\lambda/A} \cosh x\)
    for \(\lambda \ge 0\) and
    \(h(x) = \pm \sqrt{-\lambda/A} \sinh x\)
    and \(h(x) = \pm \sqrt{-\lambda/A} \sin x\)
    for \(\lambda \le 0\).
    We conclude that the metric is given by
    \[
        g = \frac{1}{f(x)^2} \biggl(
            \frac{dx^2}{A} + \frac{dy^2}{B} + A \, ds^2 + B \, dt^2
        \biggr)
    \]
    up to translating, rescaling, and swapping \(x\) and \(y\),
    where \(f \in \{\cosh, \sinh, \sin\}\).
\end{proof}

\section*{Acknowledgments}

I sincerely thank my mentor Benjy Firester
for suggesting this topic and guiding me through the research project,
always providing valuable and timely feedback.
I would like to thank the MIT PRIMES research program
for providing me with this research opportunity.
I would also like to thank Raphael Tsiamis
for helpful discussions.

\printbibliography

@article{ACF15,
  author    = {Spyros Alexakis and Dezhong Chen and Grigorios Fournodavlos},
  title     = {Singular Ricci Solitons and Their Stability under the Ricci Flow},
  journal   = {Communications in Partial Differential Equations},
  volume    = {40},
  number    = {12},
  pages     = {2123--2172},
  year      = {2015},
  publisher = {Taylor \& Francis},
  doi       = {10.1080/03605302.2015.1081609},
  url       = {https://doi.org/10.1080/03605302.2015.1081609}
}

@article{App23,
  title     = {Eguchi--Hanson Singularities in $U(2)$-Invariant Ricci Flow},
  volume    = {6},
  issn      = {2524-7182},
  url       = {https://dx.doi.org/10.1007/s42543-022-00048-y},
  doi       = {10.1007/s42543-022-00048-y},
  number    = {1},
  journal   = {Peking Mathematical Journal},
  publisher = {Springer Science and Business Media LLC},
  author    = {Appleton, Alexander},
  year      = {2023},
  pages     = {1--141}
}

@article{Bam18,
  author    = {Richard Bamler},
  title     = {Convergence of Ricci flows with bounded scalar curvature},
  volume    = {188},
  journal   = {Annals of Mathematics},
  number    = {3},
  publisher = {Department of Mathematics of Princeton University},
  pages     = {753--831},
  year      = {2018},
  doi       = {10.4007/annals.2018.188.3.2},
  url       = {https://doi.org/10.4007/annals.2018.188.3.2}
}

@article{BCCD24,
  title     = {A New Complete Two-Dimensional Shrinking Gradient Kähler--Ricci Soliton},
  volume    = {34},
  issn      = {1420-8970},
  url       = {https://dx.doi.org/10.1007/s00039-024-00668-9},
  doi       = {10.1007/s00039-024-00668-9},
  number    = {2},
  journal   = {Geometric and Functional Analysis},
  publisher = {Springer Science and Business Media LLC},
  author    = {Bamler, Richard H. and Cifarelli, Charles and Conlon, Ronan J. and Deruelle, Alix},
  year      = {2024},
  month     = 2,
  pages     = {377--392}
}

@article{BCD+22,
  title    = {Four-dimensional steady gradient Ricci solitons with $3$-cylindrical tangent flows at infinity},
  journal  = {Advances in Mathematics},
  volume   = {401},
  pages    = {108285},
  year     = {2022},
  issn     = {0001-8708},
  doi      = {https://doi.org/10.1016/j.aim.2022.108285},
  url      = {https://www.sciencedirect.com/science/article/pii/S0001870822001013},
  author   = {Richard H. Bamler and Bennett Chow and Yuxing Deng and Zilu Ma and Yongjia Zhang}
}

@misc{Cao09,
  title         = {Recent Progress on Ricci Solitons},
  author        = {Huai-Dong Cao},
  year          = {2009},
  eprint        = {0908.2006},
  archiveprefix = {arXiv},
  primaryclass  = {math.DG},
  url           = {https://arxiv.org/abs/0908.2006}
}

@misc{FT24,
  title         = {Cohomogeneity two Ricci solitons with sub-Euclidean volume},
  author        = {Benjy Firester and Raphael Tsiamis},
  year          = {2024},
  eprint        = {2408.13982},
  archiveprefix = {arXiv},
  primaryclass  = {math.DG}
}

@article{Hui24,
  title   = {Existence of singular rotationally symmetric gradient Ricci solitons in higher dimensions},
  volume  = {67},
  doi     = {10.4153/S0008439524000237},
  number  = {3},
  journal = {Canadian Mathematical Bulletin},
  author  = {Hui, Kin Ming},
  year    = {2024},
  pages   = {842--859}
}

\end{document}